\documentclass{amsart}

\usepackage{amsmath}
\usepackage{amsthm}
\usepackage{amssymb}
\usepackage{amsfonts}

\newtheorem{theorem}{Theorem}
\newtheorem{lemma}[theorem]{Lemma}

\newtheorem{definition}[theorem]{Definition}

\newtheorem{example}[theorem]{Example}
\newtheorem{examples}[theorem]{Examples}
\newtheorem{remark}[theorem]{Remark}
\newtheorem{remarks}[theorem]{Remarks}

\newtheorem{recall}[theorem]{Recall}
\newtheorem{algorithm}[theorem]{Algorithm}

\def \bbF {{\mathbb F}} 
\def \bbZ {{\mathbb Z}}
\def \bbQ {{\mathbb Q}}

\def \bbF {{\mathbb F}}
\def \bbP {{\mathbb P}}
\def \magma {{\sc magma}}
\def \Sym {\text{\rm S}}

\def \Res {\text{\rm R}}
\def \Stab {\text{\rm Stab}}

\def \calB {\mathcal{B}}

\def \Gal {\text{\rm Gal}}
\def \lcm {\text{\rm lcm}}
\def \gcd {\text{\rm gcd}}
\def \Frob {\text{\rm Frob}}
\def \disc {\text{\rm Disc}}
\def \pr {{ n}}

\begin{document}


\title[Subfields and Galois groups]%
{Computing subfields of number fields and applications to Galois group computations}

\author{Andreas-Stephan Elsenhans}
\address{Universit\"at Paderborn,
Fakult\"at EIM,
Institut f\"ur Mathematik,
Warburger Str.~100,
33098~Paderborn,
Deutschland }

\author{J\"urgen Kl\"uners}
\address{Universit\"at Paderborn,
Fakult\"at EIM,
Institut f\"ur Mathematik,
Warburger Str.~100,
33098~Paderborn,
Deutschland }

\begin{abstract}
A polynomial time algorithm to find generators of the lattice of 
all subfields of a given number field was given in~\cite{HKN}. 

This article reports on a massive speedup of this algorithm.
This is primary achieved by our new concept of {\it Galois-generating subfields}.
In general this is a very small set of subfields that determine all other 
subfields in a group-theoretic way. We compute them by targeted 
calls to the method from~\cite{HKN}. 
For an early termination of these calls, we give a list of 
criteria that imply that further calls will not result in additional subfields.

Finally, we explain how we use subfields to get a good starting group for the 
computation of Galois groups. 
\end{abstract}


\maketitle
  
\section{Introduction}
Given a number field $L  / \bbQ$ of degree $n$, we can ask for a description of all 
 fields $K$ such that $\bbQ \subset K \subset L$. 
As the number of subfields in the
multi-quadratic extension $L_e := \bbQ(\sqrt{d_1},\ldots,\sqrt{d_e})$
is not bounded by a polynomial function in $n=2^e = [L_e : \bbQ]$,
there can not be a polynomial time algorithm for this task.

However, there are less than $n$ subfields $K_1,\ldots,K_m$, such that 
all other subfields are intersections of some of these fields.
Any list of subfields with this property is called a list of 
intersection-generating subfields.
Initially, these fields were called generating subfields, but we prefer to use 
the more precise term intersection-generating as there are other operations 
(e.g.\ composition) that can construct a subfield out of others. 

A subfield $\bbQ[\alpha]$ of $L$ is described by the minimal polynomial of $\alpha$
and an embedding $\alpha \mapsto h(\beta) \in L$ for some polynomial 
$h \in \bbQ[Y]$. Here, $\beta$ denotes the primitive element of $L$ that is used
to represent $L$.
The key point of~\cite{HKN} is that we can find 
the subfield by using linear algebra. 
For efficiency this is done by using the LLL-algorithm~\cite{LLL}. 
From that we derive a primitive element $\alpha$, determine its minimal polynomial
and the embedding. 
At most $(n-1)$ of these steps are necessary to compute all so called principal 
subfields.  
Note, that a list of all principal subfields is intersection-generating.
The runtime depends heavily on the number of LLL-calls. 
In this note we will show
how hard cases ($n = 60,\ldots,100$) can be done with very few (usually $< 10$) 
LLL-calls.

The subfields found by these LLL-calls are not intersection-generating.
But, they are {\it Galois-generating}. This is the new term that 
we introduce in this article.
Using the Galois-generating fields we compute a group which has the same
block systems as the Galois group. 
As the block systems are in bijection to the subfields, we have a 
complete description of the subfield lattice.
The intersection generating subfields correspond to block systems with special 
properties. 
Thus, the remaining intersection-generating subfields can be computed 
much more efficiently compared to the LLL-method. 

In the above example $\bbQ(\sqrt{d_1},\ldots,\sqrt{d_e})$, the fields 
$$K_i := \bbQ(\sqrt{d_1},\ldots,\sqrt{d_{i-1}},\sqrt{d_{i+1}}\ldots,\sqrt{d_e})  
\mbox{ for } 1 \leq i \leq e$$ are Galois-generating. The fields $\bbQ(\sqrt{d_i})$
for $1 \leq i \leq e$ are Galois-generating, as well.  

A further improvement is that we skip LLL-calls that would
reproduce subfields that are already known. To detect this, we give a
list of criteria in Algorithm~\ref{lattice_test}.
 
All algorithms are of a $p$-adic nature. That means we work with
$p$-adic root approximations in a $p$-adic splitting field of $f$, the
minimal polynomial of a primitive element $\beta$ of $L/\bbQ$.  The
prime $p$ is chosen by the algorithms.

To derive the other subfields from the Galois-generating ones we have to compute
the intersection of certain wreath products. For this we give a graph-theoretic 
algorithm in section~\ref{sec:int}. The resulting group will be an overgroup of 
the Galois group of $f$. Thus, it can be used as a starting group for the 
computation of the Galois group of $f$ by the Stauduhar~\cite{St,Ge,FK} method.
In section~\ref{sec_applications} we give an algorithm to refine this starting group.

All the algorithms described here are available in \magma~2.23~\cite{BCP}.

\subsubsection*{Composition-generating subfields}
As an algebraic structure, the subfields form a lattice. Thus, one could ask for other ways of 
generating this lattice. For example one could ask for a set of subfields, that 
generate all subfields by composition. We suggest to call these composition-generating 
subfields. 
As an application, one can construct the maximal cyclotomic, abelian, or normal subfield 
out of these easily. 
But, it seems to be a challenging question that is almost independent of this investigation.

\subsubsection*{Acknowledgement}
Many thanks to N.\ Sutherland for writing the first \magma\ implementation 
of the~\cite{HKN} method that served as a 
starting point of this investigation.

\section{Notation}

In this article, we will consider the number field $L = \bbQ(\beta)$ 
with minimal polynomial $f \in \bbZ[x]$. 
The letter $n$ denotes the degree of $L$.
We denote the normal closure of $L$ by $\tilde{L}$.
The roots of $f$ in $\tilde{L}$ are denoted by $\beta = \beta_1,\ldots,\beta_n$.
For simplicity we assume the $\beta_i$ to be algebraic integers. 

The letter $K$ denotes a subfield of $L$. The minimal polynomial of a
primitive element $\alpha$ of $K$ is called $g$. In the case that we are dealing with 
several subfields we add indices to $K$, $\alpha$, and $g$. 

We will use the term Galois group in the following way: The Galois group of
a polynomial $f$ is the Galois group of a splitting field represented as a permutation 
group acting on the indices of the roots of $f$. The Galois group
of a number field is the Galois group of the minimal polynomial of a primitive element. 

Any permutation group that is an overgroup of the Galois group of $f$ is called a 
{\it Galois starting group} of $f$. 

When we give examples that involve permutation groups, we write $n$T$k$ for 
the $k$-th transitive group of degree $n$ in the database of transitive 
groups~\cite{Hul,CH}. Currently this is available in \magma\ for $n < 48$. 

\section{Galois theory of subfields and block systems}

\begin{recall}(Partitions) \label{recall_blocksys}
Let $G$ be a transitive permutation group of degree $n$. We will
use the following terms and facts~\cite{DM,HEB, Hup}.
\begin{enumerate}
\item \label{blocks_def}
A partition or system of blocks $\calB = \{B_1,\ldots,B_k\}$ is a set of disjoint, 
non-empty subsets of $\{1,\ldots,n\}$ that cover all of $\{1,\ldots,n\}$ 
such that $\sigma B_i \in \calB$ for all $\sigma \in G$ and 
$i = 1,\ldots,k$. The $B_i$ are called blocks.
\item
Let $\calB_1$ and $\calB_2$ be two systems of blocks. In the case that 
each block of 
$\calB_1$ is a subset of a block of $\calB_2$, we call $\calB_1$ finer than
$\calB_2$.
\item
There is a 1--1 correspondence between the super-groups of the stabilizer 
of $1$ and the block systems. It is given by mapping a system of blocks 
to the stabilizer of the block containing 1~\cite[Theorem 1.5A]{DM}.
\item
The largest subgroup of $\Sym_{\ell k}$ that has a system of $\ell$ blocks of size $k$
is the wreath product $\Sym_k \wr \Sym_ \ell$. It is isomorphic to the
semi-direct product $\Sym_k^\ell \rtimes \Sym_\ell$. The action on $\Sym_k^\ell$ is given 
by permutation of the components.
\end{enumerate}
\end{recall}

\begin{recall}(Partitions and $G$-congruences)
It is well known that there is a 1--1 correspondence between all partitions of a set and all
equivalence relations defined on the set~\cite[Chap. 6.5 Th. 2]{Rosen}. Thus, all statements
about block systems can be translated into the language of equivalence relations.
Let $\calB = \{B_1, \ldots, B_k \}$ be a block system and $\sim$ the corresponding relation.
\begin{enumerate}
\item
We have $j_1 \sim j_2$ if and only if $\exists i \in \{1,\ldots, k\} \colon j_1, j_2 \in B_i$.
I.e., the equivalence classes of the relation are exactly the blocks.
\item
The $G$-compatibility of the partition (Recall~\ref{recall_blocksys}.\ref{blocks_def})  
turns $\sim$ into a relation with
$$
\forall j_1,j_2 \in \{1,\ldots,n\}, g \in G \colon j_1 \sim j_2 \Longrightarrow g j_1 \sim g j_2.
$$
An equivalence relation with this property is called a $G$-congruence~\cite[Def. 2.26]{HEB}.
Conversely, the equivalence classes of a $G$-congruence are a system of blocks for~$G$.
\item
For each subset $S \subset \{1,\ldots,n\}^2$ there is a $G$-congruence $\sim$ generated by 
$S$~\cite[Def. 2.27]{HEB}. This is the smallest $G$-congruence $\sim$ with
$$
\forall (j_1,j_2) \in S \colon j_1 \sim j_2\, .
$$
\item
A $G$-congruence and the corresponding system of blocks are called principal, 
if the $G$-congruence is generated by one pair.
\item
As $G$ is transitive, we can generate each principal $G$-congruence by $\{(1,j)\}$.
Thus, there are at most $n-1$ principal partitions.
\item
Let $\calB$ be a system of blocks and $\sim$ be the corresponding $G$-congruence.
Further, denote by $B_1 \in \calB$ be the block containing 1.
Then the principal $G$-congruences that are finer than $\sim$ are exactly those
that are generated by $(1,i)$ for $i \in B_1, i \not= 1$.

Further, $\sim$ is the transitive closure of the union of all 
these principal $G$-congruences.
\item
More generally, let $\calB_1,\ldots,\calB_r$ be some systems of blocks.
Then there is a finest system of blocks $\calB$, such that all the 
$\calB_i$ are finer than $\calB$.
The $G$-congruence corresponding to $\calB$ is generated by
$$
\{(j_1, j_2) \colon j_1,j_2 \in B, B \in \calB_1 \cup \ldots \cup \calB_r \}
\, .
$$
In other words, it is the transitive closure of the 
union of all the $G$-con\-gru\-ences corresponding to the $\calB_i$.
\end{enumerate}
\end{recall}

\begin{recall}(Relations and graphs)
A relation $R$ on a finite set $M$ can be encoded in a graph by taking $M$ as the 
set of vertices. The directed edge $(a,b)$ is part of the graph if and only if $(a,b) \in R$. 
As we are only interested in symmetric relations, we can work with undirected graphs. 
\begin{enumerate}
\item
Let $\sim$ be an equivalence relation on $M$ and $G$ be the corresponding graph. 
The vertices of the connected components of $G$ are the equivalence classes of~$\sim$.
\item
Let $\sim$ be a symmetric relation on $M$ and $G$ be the corresponding graph. 
The vertices of the connected components of $G$ are the equivalence classes of 
the smallest equivalence relation containing $\sim$~\cite[Chap. 6.4]{Rosen}.
\item
Thus, we can use graph algorithms to compute connected components of undirected 
graphs \cite[Chap 22.5]{CLRS} to handle the transitive closure of a given relation.
\end{enumerate}
\end{recall}

\begin{remark}
The main theorem of Galois theory gives us a 1--1 correspondence between the 
intermediate fields of a normal closure of a field extension and the 
subgroups $U$ of its Galois group. Two subgroups $U_1 \subset U_2$ correspond 
to two subfields $K_2 \subset K_1$. I.e., the correspondence is 
inclusion-reversing.

We view the Galois group $G$ of $L$ as a subgroup of $\Sym_n$ 
that acts on the indices of the roots of $f$. Then the stem field $L$ 
corresponds to the stabilizer of the index of the first root. 
This shows that we have a 1--1 
correspondence between the block systems of $G$ and the intermediate fields 
of $L / \bbQ$.
\end{remark}

\begin{remark}
The principal subfields as defined in~\cite[Def. 3]{HKN} correspond 
to the principal block systems defined above. In the notion of~\cite{HKN} 
the principal subfields are a generating set. The fields in the (unique) 
smallest generating set are called the generating subfields. In this 
article we use the principal subfields as an (intersection-)generating set.
We do not aim for the smallest intersection-generating set. It is easy to
describe it in relation theoretic terms. 
\end{remark}

We can use the following algorithm to compute all the principal
$\Gal(\tilde{L}/\bbQ)$-congruences.
\begin{algorithm}[Principal systems of blocks by factoring]
Denote by $\beta_1,\ldots,\beta_n \in \tilde{L}$ the roots of the polynomial defining $L$.
\begin{enumerate}
\item
Find a shift $s \in \bbZ$, such that 
$\Res_2(X) := \prod_{i < j} (X - (\beta_i + s)(\beta_j + s))$  is squarefree.
Compute $\Res_2$ with the method given in \cite{BFSS} without computing the $\beta_i$.
\item
Factor $\Res_2$.
\item \label{identify_res2_orbits}
For each irreducible factor $F$ of $\Res_2$, identify which pairs of roots are 
encoded in it. I.e., determine $R_F := \{ (i,j) \mid F((\beta_i + s)(\beta_j + s)) = 0\}$.
\item
For each relation $R_F$, compute the reflexive, transitive, and symmetric closure. 
\item
Return the equivalence relations found as a complete list of principal 
$\Gal(\tilde{L}/\bbQ)$-congruences. 
\end{enumerate}
\end{algorithm}

\begin{remarks}\label{subfields_from_generating_subfields} 
\begin{enumerate}
\item
Assume the above algorithm results in the (principal) $G$-con\-gru\-ences $R_1,\ldots,R_k$. 
Then all other $G$-congruences
are given as the transitive closure of 
$R_{i_1} \cup \cdots \cup R_{i_m}$, for any choice of indices 
$i_1,\ldots,i_m \in \{1,\ldots,k\}$.

In other words, the principal subfields are intersection-generating.
\item
As $G$ acts transitively on $\{1,\ldots,n\}$, there is no factor of $\Res_2$
of degree less than $\frac{n}{2}$. Thus, we 
reproduce the upper bound of $(n-1)$ intersection-generating subfields.
\item
It is known that the computation of $\Res_2$ can be done in  
polynomial time~\cite{BFSS} and polynomial factorization is polynomial time 
as well~\cite{LLL}.
Thus, we have a polynomial time algorithm to compute inter\-section-generating 
subfields as soon as we can do step~\ref{identify_res2_orbits} in polynomial time 
by using approximations of the roots. 
A solution to this is given in~\cite[p. 255]{Kl}. Denote by $r_i$ 
$p$-adic approximations of the roots $\beta_i$.
Then, one can work with the smallest field which is sufficient to represent the
polynomial $\prod_{i \in B} (X - r_i)$ for a block $B$. 
Note that, the degree of the field of definition of this polynomial is bounded 
by the degree of the subfield we are looking for.

Thus, we have a polynomial time algorithm that computes all principal subfields.
\item
In case we ask for the list of all subfields instead of a list of
intersection-generating subfields we can either compute the intersections
directly, or we can use the description of the subfields in terms of 
systems of blocks. This means, we have to list all equivalence relations
that result as the transitive closure of the union
of some equivalence relations already known.
For this we suggest to use an approach similar to~\cite[Sec.~2.2]{HKN}. The advantage is, 
that run time is linear in the number of subfields.
\end{enumerate}
\end{remarks}

Assume that ($p$-adic) approximations $r_1,\ldots,r_n$ of the roots of $f$ are given.
We need algorithms to make the correspondence between
subfields and block systems explicit. For this we first introduce block invariants.

\begin{definition}
Let $M$ be a nontrivial subset of $\{1,\ldots,n\}$. 
A polynomial $I \in \bbZ[X_1,\ldots,X_n]$ is called a {\it block invariant} for $M$,
if $I$ only involves the variables $\{X_k : k \in M\}$ and is invariant under
all permutations of the elements of $M$. 
\end{definition}

\begin{example}
A nontrivial block invariant of smallest degree is given by $\sum_{i \in M} X_i$.
More generally, we can use $\sum_{i \in M} T(X_i)$ with any transformation 
$T \in \bbZ[X]$. 

Other block invariants are given by 
$\prod_{i \in M} X_i$ and $\prod_{i \in M} (X_i + s)$
for any $s \in \bbZ$.
\end{example}

\begin{remark}
In the language of invariant theory, a block invariant $I$ is also called a {\it relative invariant}
for $\Stab_{\Sym_n}(M) \subset \Sym_n$. The interested reader may consult~\cite{E2}
for more constructions of relative invariants.
\end{remark}

\begin{remark}[Block system to a subfield]\label{blocks_from_subfield}
Let $L / K / \bbQ$ be a tower of fields. Further, 
let $h(\beta)$ be a primitive element of the degree $m$ subfield $K$. 

In this situation $i$ and $j$ are in the same block (of the block system 
corresponding to $K$) if and only if $h(\beta_i) = h(\beta_j)$. In the case that we 
work with root approximations $r_1,\ldots,r_n$ we can 
use $h(r_i) \not= h(r_j)$ implies $h(\beta_i) \not= h(\beta_j)$,
but 
if the chosen precision is not sufficient, it might be that 
$h(r_i)=h(r_j)$ and $h(\beta_i)\not=h(\beta_j)$.
As the total number of blocks is $m$ and each 
block has $n/m$ elements, this suffices to determine the block system. 
\end{remark}

The explicit description of the inverse operation is more complicated as we 
have to compute the minimal polynomial of a primitive element $\alpha$ of 
the subfield and its image in $L$.

\begin{algorithm}[Subfield from block system]\label{subfield_from_blocks}
Given a system of blocks $\calB = \{B_1,\ldots,B_m\}$ for the 
root approximations $r_1,\ldots,r_n$ of $f$. We assume the roots $\beta_1,\ldots,\beta_n$ 
of $f$ to be algebraic integers.
This algorithm computes a defining polynomial $g$ for the corresponding 
subfield by using only the approximations $r_i$ in some extension of $\bbQ_p$.
\begin{enumerate}
\item
Compute a bound $C$ of the absolute values of the roots of $f$. 
(E.g., the Fujiwara-bound.)
\item
Find a non-degenerated block invariant 
$I = \prod_{i \in B_1} (X_i - s) \in \bbZ[X_1,\ldots,X_n]$ 
for some $s\in\bbZ$. Here, non-degenerated means, that the values 
$I_k := \prod_{i \in B_k} (\beta_i - s)$ on all blocks are pairwise distinct. 
In the case that $p > n^2$, we can even find an $s$
such that these products are different in the residue class field 
of the $p$-adic splitting field~\cite[Lemma 42]{Kl}.
\item
The $I_k$ are algebraic integers. All complex conjugates are bounded by $C' := (C+s)^{\#B_1}$. 

The coefficients of $g := \prod_{k=1}^m (X - I_k)$ are integers with absolute 
value at most $C'' := (C'+1)^m$. 
\item
Compute $p$-adic approximations of the roots $r_i$, the values of $I_k$ and the 
coefficients of $g$ with precision at least 
$\lceil \log_p(2 C'') \rceil$. 
\item
Reconstruct $g \in \bbZ[X]$ from its $p$-adic approximation.
\end{enumerate}
\end{algorithm}

The embedding is given as the solution of the interpolation problem
$h(r_i) = I_k$ for all $i \in B_k$ and all $k \in \{1,\ldots,m\}$.
Here, $h$ is a polynomial of degree at most $(n-1)$ with rational coefficients. 
It defines the embedding $\iota \colon K \rightarrow L$ via 
$\iota(\alpha) = h(\beta)$. 

The conditions $h(r_i) = I_k$ for all $i,k$ with $i \in B_k$ above encode an 
interpolation problem that translates to a uniquely solvable linear 
system for the coefficients of $h$.
To solve it efficiently, we use Newton-lifting as follows.

\begin{algorithm}[Embedding of subfield] \label{subfield_embedding}
Let $L / K /\bbQ$ be a tower of fields  and the block invariant used to construct $K$
be given. This algorithm computes the embedding of $K$ in $L$ by determining~$h(\beta)$.
\begin{enumerate}
\item
Solve the interpolation problem in the residue class field of the $p$-adic splitting field
directly. As we assume the roots to be pairwise distinct in the residue class field
the system has still a unique solution.
We get $h_0 \in \bbZ[X]$ with
$g(h_0(\beta)) \equiv 0 \bmod p$.   
\item
Find $v_0 \in \bbZ[\beta]$ with $v_0 g'(h_0(\beta)) \equiv 1 \bmod p$.
\item
Lifting these initial $p$-adic approximations results in the sequences
\begin{eqnarray*}
h_{n} &:=& \left(h_{n-1} -  g(h_{n-1}) \cdot v_{n-1}\right) \bmod p^{2^n} \\ 
v_{n} &:=& \left(v_{n-1} - (g'(h_{n-1}) \cdot v_{n-1} - 1) v_{n-1}\right) \bmod p^{2^n} \, .
\end{eqnarray*}
\item
In general, only $h \cdot f'(\beta)$ is known to be an element of the
equation order $\bbZ[\beta]$.  
Thus, we compute $h_n \cdot f'(\beta) \bmod p^{2^n}$. 
In the case that 
all coefficients can be represented by integers of absolute value (significantly) 
smaller than $\frac{1}{2} p^{2^n}$ (e.g., smaller than $10^{-6} p^{2^n}$), 
we assume this representative coincides with $h \cdot f'(\beta)$. 
In this case, a single division results in a guess for $h$.
Otherwise we increment $n$ and go back to step 3 to continue the lifting process.
\item
As soon as we have a guess for $h$ we check it by testing $g(h(\beta)) = 0$.
If this check is passed, we return $h$ as the embedding. Otherwise, we
increment $n$ and go back to step 3 to continue the lifting process.
\end{enumerate}
\end{algorithm}

\begin{remark}
When computing the final check $g(h(\beta)) = 0$ an interesting phenomenon occurs. 
If the check is passed, the computation is a lot faster compared to the case of a failure.

This is explained by the fact that the correct $h(\beta)$ and all the intermediate results 
of the computation of $g(h(\beta))$ are algebraic integers.
Therefore, all denominators are 
bounded by the index of the equation order in the maximal order. 
In the case of an incorrect guess, there is no reason for a denominator bound. 

That explains why we do the final check only in the case that we have a good 
reason to believe that we are already correct.
\end{remark}
 
\section{Using cycle types \label{cycle_section}}
Given an irreducible polynomial $f \in \bbZ[X]$ and a prime number $p$ that 
does not divide the discriminant of $f$, we can easily compute the roots of 
the reduction of $f$ modulo $p$ in an extension of $\bbF_p$ together with
the action of the Frobenius on the roots. Is is well known, that this 
permutation is an element of the Galois group of $f$ that can be identified 
up to conjugation. Further, by Chebotarev's density theorem
each conjugacy class of the Galois group occurs for infinitely many choices of $p$. 

If we would be able to match the roots of the reductions modulo different 
primes, we would be able to generate the Galois group of the
splitting field heuristically. 

This indicates, that we can get some information about the Galois group 
by inspecting the factorization of $f$ modulo several primes. But, as we do not 
know how to identify roots of different reductions we have to use indirect 
approaches to get something out of this. 

Currently, we are only able to use the cycle type of the Frobenius element. This
coincides with the degrees of the irreducible factors of $f$ modulo $p$. 
Note that fixed points are always included in a cycle type.

\subsection{A divisor for the order of the Galois group} \label{divisor_of_order}
Our algorithm will need a lower bound of the order of the Galois group. For this we will
use the following.
\begin{lemma}
Let $(n_1,\ldots,n_j)$ be the cycle type of an element of the transitive
permutation group $G \subset \Sym_n$. Then
$$
n \frac{\lcm(n_1,\ldots,n_j)}{\gcd(n_1,\ldots,n_j)} \mid \#G \,. 
$$
\end{lemma}

\begin{proof}
Denote by $U$ the index $n$ subgroup of $G$ that stabilizes one point and by 
$\sigma$ an element of $G$ with cycle type $(n_1,\ldots,n_j)$.
Then $\sigma^{n_i}$ has at least one fixed point. Therefore, it is contained in
a subgroup conjugate to $U$. Thus, the order of $\sigma^{n_i}$ is a divisor of $\#U$.
This shows
$$
\frac{\lcm(n_1,\ldots,n_j)}{n_i} \mid \# U
\,. 
$$
Thus,
$$
\lcm \left\{\frac{\lcm(n_1,\ldots,n_j)}{n_i} : i = 1,\ldots,j \right\}
= \frac{\lcm(n_1,\ldots,n_j)}{\gcd(n_1,\ldots,n_j)}
$$
is a divisor of $\#U$ as well. As $U$ is of index $n$ in $G$ the claim is proved.
\end{proof}

The lemma above uses only  one cycle type and the fact that $f$ is irreducible to derive a divisor 
of the group order. In practice, we apply the lemma to several (e.g. $n$) primes and form the lcm of
all divisors found.

To do even better we have to combine several cycles in a more subtle way. 
This is done by the following lemma.

\begin{lemma}
Let $G \subset \Sym_n$ be a $p$-group and $\sigma,\tau \in G$ be two elements of order $p^e, p^f$.
Assume that the numbers of points with trivial stabilizer 
(i.e.,\ with orbit of length $p^e$ resp.\ $p^f$) in $\langle \sigma \rangle$ and 
in  $\langle \tau \rangle$  do not coincide, then the order of $G$ is divisible by $p^{e+f}$.
\end{lemma}

\begin{proof}
Let $e,f$ with $e \geq f$ be two integers with $e+f$ minimal such that the claim is not correct. 
We have $e,f > 0$ as otherwise the claim of the lemma is trivial. 

First observe that $\langle \sigma \rangle$ and $\langle \tau \rangle$ are cyclic $p$-groups. 
Thus, they have a unique minimal subgroup. This shows that the set of points with trivial 
stabilizer is unchanged when we switch to a nontrivial subgroup. 

In the case that $e > f$ we take a point $x$ with trivial stabilizer in $\langle \sigma \rangle$. 
This point has an orbit $M := \langle \tau \rangle x$. Now we inspect the subgroup $U := \Stab_G(M)$.
Obviously we have $\tau \in U$ and $\sigma \not \in U$. Let $k$ be the smallest integer 
such that $\sigma^{p^k} \in U$. Then we have $k > 0$ and $p^k | [G:U]$. 
In the case $k = e$ we are done. Otherwise $U$ contains the nontrivial element 
$\sigma^k$ of order $p^{e - k}$.  As the points with trivial stabilizer 
in $\langle \sigma \rangle$ and in $\langle \sigma^{p^k} \rangle$ 
coincides, we can conclude $p^{f + e - k} | \# U$ and $p^{e + f} | \# G$. 

In the case $e = f$ we assume that the number of points with trivial stabilizer for  
$\langle \sigma \rangle$ is bigger than for $\langle \tau \rangle$. Thus, there is a 
point $x$ having an orbit of length $p^e$ with respect to $\langle \sigma \rangle$ 
and another (shorter) orbit with respect to $\langle \tau \rangle$.
We choose $M := \langle \tau \rangle x$. Now we inspect the stabilizer $U := \Stab_G(M)$. 
We continue as above. We get $\tau \in U$ and $\sigma \not \in U$. 
Setting $k$ to the smallest integer with $\sigma^{p^k} \in U$ we conclude 
$k > 0$, $p^k | [G : U]$ and $p^{e + f - k} | \#U$. Thus, the claim is proven.
\end{proof}

\begin{remark}
In the case of a transitive group $G \subset \Sym_n$ the above lemma can be applied to all 
cycles of $p$-power order having a fixed point. Then we can determine a divisor of 
$U := \Stab_G(1)$. As this is a subgroup of index $n$. We can multiply the
divisor of the order of $\#U$ by $n$ and still get a divisor of $\#G$.
\end{remark}

\subsection{Exclusion of block sizes}

Is is well known, that the block size of a block system of a subgroup of $\Sym_n$ is a 
divisor of $n$. A priori, all divisors of $n$ are possible block sizes. 
We use the following to exclude as many block sizes as possible. If we can 
exclude all of them, we have proved that the field is primitive. Thus, we can skip 
the entire subfield search.
As each permutation group with a block of size $k$ is contained in a wreath product
of the form $\Sym_k \wr \Sym_\ell$ we can use the following lemma.

\begin{lemma}\label{lemma_block_size}
Let $\sigma \in \Sym_k \wr \Sym_\ell$ be an element with cycle type $(n_1,\ldots,n_j)$. 
Then, for each index $m = 1,\ldots,j$
there is an integer $e$ and an index set $I_m \subset \{1,\ldots,j\}$ 
with $m \in I_m$ such that
$$
\forall i \in I_m \colon e \mid n_i 
~\mbox{and}~ \sum_{i \in I_m} n_i = e \cdot k\, .
$$
\end{lemma}

\begin{proof}
Denote the blocks that contain the elements from the $m$-th orbit of $\sigma$ by
$B_1,\ldots,B_e$. As $\sigma$ acts transitively on the $m$-th orbit is has to 
act transitively on the blocks $B_1,\ldots,B_e$ as well.

Further, all elements in the blocks $B_1,\ldots,B_e$ are in orbits of 
size divisible by $e$. Let $I_m$ be an indexing set for these orbits. 
We have $m \in I_m$ trivially. 
By counting the points in these blocks we get $\sum_{i \in I_m} n_i = e \cdot k$
and all summands are divisible by $e$.
\end{proof}

\begin{remark}
Lets inspect the particular case of a fixed point in the lemma above closer. 
This results in $e=1$ and encodes that the fixed point is contained in a fixed block.
Thus, the size of the block must be the sum of the length of some orbits of 
$\sigma$ that includes the fixed point.
\end{remark}

\section{Intersection of wreath products}\label{sec:int}

\begin{remark}
As explained above, knowing a subfield of a number field is equivalent 
to knowing a block system of its Galois group. 
The latter means that the Galois group is contained in
a certain wreath product. 

As soon as we know more than one subfield, we know that the Galois group
is contained in the intersection of several wreath products. 
Thus, we need an efficient algorithm to compute the intersection. 
\end{remark}

\begin{algorithm}[Intersection of wreath products]\label{algo_inter} 
Let $\calB_1, \ldots, \calB_s$ be systems of blocks of the 
underlying set $\{1,\ldots,n\}$. 
Denote by $n_1,\ldots,n_s$ the number of blocks of the systems.
\begin{enumerate}
\item
Build a graph with a total number of $s + n_1 + \cdots + n_s + n$ vertices.
\item
Fix a bijection between the vertices and the systems of blocks, the blocks in all
 block systems and the $n$ elements of the underlying set.
\item
Add an edge between a point-vertex and a block vertex if the point is
contained in the block. Add an edge between a block vertex and a block systems
vertex when ever the block is contained in the system of blocks.  
\item \label{color_step}
Color the vertices representing the $s$ block systems with $s$ different colors.
\item
Compute the automorphism group of the colored graph by using {\tt nauty} \cite{KP}.
\item
Compute the action of the automorphism group on the vertices representing the $n$ points.
\item
Return the image of the action as the intersection of the wreath products 
corresponding to the systems of blocks given.
\end{enumerate}
\end{algorithm}

\begin{remarks}
\begin{enumerate}
\item
To illustrate the advantage of the graph approach in contrast to the use of a 
general intersection algorithm, we construct permutation groups of degree 80 
out of the 245 groups of order $17 \cdot 80$ by taking the coset-action 
with respect to a 17-Sylow subgroup.
We compute all principal systems of blocks of all these groups and build 
the corresponding wreath products. Using the standard intersection 
algorithm in \magma, we intersect all wreath products constructed.
The total time for the intersections is 115 seconds. 
Using the graph-theoretic approach we can compute it in 6.5 seconds.
\item
Doing the same in degree 165 with the 181 groups of order $8 \cdot 3 \cdot 5 \cdot 11$ 
by acting on the cosets of a 2-Sylow subgroup 
results in a total intersection time of 1251 seconds. 
The graph theoretic approach takes only 1.6 seconds.
\item
There is no good worst case bound for the run time of {\tt nauty} as an unpredictable
amount of time can be used in the recursive search in dead ends. 
A systematic test with all graphs that this approach constructs 
out of the database of transitive groups 
up to degree 32 shows, 
that the back-track search of {\tt nauty} never runs into a dead end. 
\item
In case we start Algorithm~\ref{algo_inter} with all principal systems of blocks of 
a permutation group and uses only one color in step \ref{color_step}, we get
the normalizer of the intersection of the wreath products. 
When we try this, the run-time of {\tt nauty} is much larger.
\item
Because of all of this, we believe all the graphs that are constructed 
by Algorithm~\ref{algo_inter} 
to be special. We conjecture that computing the intersection of wreath products 
in this way is a polynomial time algorithm.
\end{enumerate}
\end{remarks}

\section{Galois generating subfields}

\begin{definition}
A set of subfields $(K_i)_{i \in I}$ of $L$ is called Galois-generating,
if all  block systems of the Galois group of $L$ are block systems 
of the intersection of the wreath products corresponding to $(K_i)_{i \in I}$.
\end{definition}

\begin{remark}
As we can express the intersection of subfields purely in terms of block systems,
each intersection-generating set of subfields is Galois-generating as well. 
In many examples we observe that a minimal set of Galois-generating subfields 
is far smaller than a minimal set of intersection-generating subfields. 
We will use this to compute the subfields efficiently. 
\end{remark}

\begin{example}
Let $L$ be a field of degree $n = 2^m$ with elementary abelian Galois group ${\rm C}_2^m$. 
Each pair of roots of $f$ corresponds to a principal partition of block size $2$. 
If we want to find all these principal subfields directly with the LLL-approach, 
we need $n-1$ calls of the LLL-algorithm. 

However, when we know two of these principal subfields, the intersection $K$ of these
two fields results in a relative degree 4 extension $L/K$ with two known subfields. Thus,
the relative Galois group has to be the Kleinian four group. Thus, a third principal 
subfield is for free.

Calling the LLL once more gives us an new principal subfield with block size two. 
Inspecting each pair of the new and a known principal subfield in the same way as 
above gives us another 3 subfields for free. 

This can be continued. We get all principal subfields out of 
$m$ Galois-generating fields. Thus, $m$ instead of $2^m-1$ LLL-calls are sufficient.
\end{example}

\begin{remark}
In the above example we used the structure of the Galois group to explain how new principal
subfields are caused by others. In general the combinatorics can be far more complicated.
That is why we take the detour via the block systems of the intersection of wreath products. 
\end{remark}

The outline of our subfield algorithm is as follows.

\begin{algorithm}[Outline of field search]
To compute the systems of blocks and the subfields of a field $L$ we do the following:
\begin{enumerate}
\item
Factor the defining polynomial of $L$ modulo various primes. 
Derive as much information as possible from the cycles found. 
In the best case this can prove that no subfields exist.
\item
Select a prime $p_s$ for a $p_s$-adic splitting 
field and a prime $p$ for the LLL-computations. 
\item
Loop over the $p$-adic factors of $f$.
\item
Test if the factor may result in a new principal system of blocks.
\item
If a new principal system of blocks can not be excluded, 
compute the corresponding subfield with the LLL-approach and label 
it as proved to be principal.
\item
If a new subfield is found make its system of blocks explicit in 
$p_s$-adic arithmetic and compute a Galois starting group $G$. 
E.g., one can take $G$ as the intersection 
of the wreath products corresponding to all known block systems.
\item
If the group $G$ has more principal systems of blocks
than known, add these to the list of known block systems.
\item
If we can derive from the Galois starting group that we are done, 
terminate prematurely.
\end{enumerate}
If we want all subfields explicitly, we can derive them from 
$G$, as the Galois group and $G$ have the same systems of blocks.
\end{algorithm}

\begin{remarks}
\begin{enumerate}
\item
Note that a system of blocks which is principal with respect to an intermediate group 
might not relate to a principal subfield. The point is, that a principal system of blocks
for one group does not need to be principal with respect to all its subgroups. 
\item
Each group has at most $n-1$ principal block systems. Thus, each subfield found  
causes the computation of at most $n-2$ additional fields that are principal 
for the current group $G$.  
As we have at most $n-1$ subfields found by LLL (and Galois starting group), 
we deal with at most $(n-1)^2$ subfields. 
Thus, this is still polynomial.
\item
In practice many LLL-calls are performed to confirm that a known subfield is in 
fact principal with respect to a certain pair of roots. We can optimize the 
algorithm by detecting this beforehand. If this is possible, we can skip the factor.
Here, we can use the lattice of already known subfields. 
An extra bit of useful information is that some of them are proven to be principal.
\end{enumerate}
\end{remarks}

\section{Detailed algorithms}

To make the steps in the above outline more explicit, we give the full field search algorithm 
and the sub-algorithms that are called.

\begin{algorithm}[Field search] \ \\
{\bf Input:} {\rm A polynomial $f$ defining the number field $\bbQ[X] / f$.} \\
{\bf Output:} {\rm A list of intersection-generating subfields.}
\begin{enumerate}
\item
Call Algorithm~\ref{prime_inspection} for initialization. 
If this shows that no subfields exist terminate.
\item
Set up a $p_s$-adic splitting field of $f$ and factor
$f = f_1 \cdots f_m$ $p$-adically with $\deg(f_1) = 1$.
\item
For each factor $f_j, j = 2,\ldots,m$ do the following:
\begin{enumerate}
\item
Call Algorithm~\ref{lattice_test} to check that there 
is space left for an additional principal block system corresponding to
$f_j$. If not continue with the next factor.
\item
Call Algorithm~\ref{principal_subfield} with the factors $f_1,f_j$ 
to get a principal subfield.
\item
Compute the system of blocks corresponding to the subfield found as described in 
Remark~\ref{blocks_from_subfield} in $p_s$-adic arithmetic.
Label it as proven to be principal.
\item
If the system of blocks was already known, continue with the next factor.
\item
Add the block system to the list of known block systems.
\item
Compute the group $G$ as the
intersection of the corresponding wreath-products by using Algorithm~\ref{algo_inter}. 
If $\Gal(f) \subseteq A_n$ is known, intersect $G$ with $A_n$.
\item \label{prinzipal_explizit}
If $G$ has more principal block systems than known add these block 
systems and the corresponding subfields 
to the list of known block systems and subfields.
\item
If the order of $G$ equals 
the divisor of the group order known, terminate the loop.
\item
In case the order of $G$ is twice the lower bound of 
the group order, call Algorithm~\ref{final_adjust}.
If an additional system of blocks is found, recompute $G$. 
Terminate the loop.
\end{enumerate}
\item \label{return_fields}
Return the list of subfields 
corresponding to the principal block systems of $G$.
\end{enumerate}
\end{algorithm}

\begin{remarks}
\begin{enumerate}
\item
We use the two primes $p$ and $p_s$, to optimize the costs of the 
$p_s$-adic arithmetic and to find the 'best' prime for the LLL-approach.
\item
In the special case $p = p_s$, there is no need to compute the corresponding
subfields in step~\ref{prinzipal_explizit}. 
If in addition, step~\ref{return_fields} would only return the fields already
computed, we would return only Galois-generating subfields.
\end{enumerate}
\end{remarks}

\begin{algorithm}[Prime inspection] \label{prime_inspection} \ \\
{\bf Input:} {\rm A polynomial $f$ defining the number field $\bbQ[X] / f$.} \\
{\bf Output:} {\rm `No subfields' or
\begin{enumerate} 
\item Possible block sizes.
\item The LLL-prime $p$ and the splitting field prime $p_s$.
\item A divisor of the order and the sign of $\Gal(f)$.
\end{enumerate} }
\begin{enumerate}
\item
Enumerate the divisors of the degree of $f$ as potential block sizes.
\item
Factor $f$ modulo several primes not dividing $\disc(f)$. Store the cycle types found.
\item
Rule out all block sizes that are excluded by Lemma~\ref{lemma_block_size} 
applied to the cycle types found.
In the case that all block sizes are ruled out return `No subfields'.
\item
Continue these steps, until at least one prime with a linear factor and
a prime $p_s$ with a reasonable splitting field degree are found.
\item
If all  cycle types found correspond to even
Frobenius permutations do the following:
Test the discriminant of $f$ to be a square. If so, note that the  
Galois group is even.
\item
Compute a divisor of the order of the Galois group from the cycle types found
by using the methods described in Section~\ref{divisor_of_order}.
\item
Out of the primes with a linear factor select the one with a minimal number of 
factors of degree less than the largest possible block size. Call this prime~$p$.
\item 
Let $p_s$ be the inspected prime that results in the smallest $p_s$-adic splitting field.
\item
Return $p$ and $p_s$ as the LLL-prime and the splitting field prime we work with.
Further, return the divisor of the group order and all  block sizes that
are not sieved out and the information whether the Galois group is even.
\end{enumerate}
\end{algorithm}

The given field search algorithm is some type of compromise. It stops the LLL-steps if either all 
subfields are known or the Galois group is determined up to a factor of 2. At this point the
idea is to look for the exact Galois group. Thus, all the index two subgroups of the known group
have to be inspected with a quick test. This is done by the next algorithm.

\begin{algorithm}[Final adjustment] \label{final_adjust} \ \\
{\bf Input:} {\rm $G \leqslant \Sym_n$ with $[G : \Gal(f)] \leq 2$, all subfield information known.} \\
{\bf Output:} {\rm Missing Galois generating subfield or 'no Galois-generating subfield missing'.}
\begin{enumerate}
\item
Compute all transitive subgroups of~$G$ of index 2.
\item
Rule out all subgroups with the same principal partitions as~$G$.
\item
For each subfield proven to be principal, rule out the subgroups 
that turn it into a non-principal subfield.
\item
For each subgroup left, pick a new principal system of blocks 
and prove or disprove that it corresponds to an existing subfield
by using Remark~\ref{block_direkt_beweisen}.
(This can succeed for at most one subgroup.)
\item
If a new subfield is found, return it and the corresponding system of blocks.
\item
Return 'no Galois-generating subfield missing'.
\end{enumerate}

\end{algorithm}

\begin{remark} 
Example~\ref{A5_Beispiel} and Example 3 in Table~\ref{test_tab} show the different
possible behaviors of Algorithm~\ref{final_adjust}. 
In the first case we descent to the right index 2 subgroup, in the second case we prove
that the input was already the right starting group. Thus, in the first case we can reduce
from 3 to 2 LLL-calls.  In the second case we can skip 10 fruitless LLL-calls.
\end{remark}

\begin{remark} \label{block_direkt_beweisen}
A method to prove that a conjectural block system is in fact a block system 
is described in~\cite[Algo. 44, 46]{Kl}. The basic idea is to compute the 
corresponding subfield by
using Algorithm~\ref{subfield_from_blocks} and~\ref{subfield_embedding}. 
But, one has to return fail in the case that either the
coefficients of the subfield polynomial are bigger than predicted by the bound
computed from the root bound or the coefficients of the embedding get too big.
For the latter we can use the bound from~\cite[Lemma 18]{HKN}.    
\end{remark}

Given two irreducible $p$-adic factors $f_1,f_2$ of $f$ with $\deg(f_1) = 1$,
we have to compute the largest subfield $K$ of $L$, such that the roots of $f_1$ and
$f_2$ are in the same block with respect to the system of blocks corresponding to $K$.

The basic idea is to use the LLL-approach as described in~\cite{HKN} to construct it and
use the same proof as above to confirm it. We give a detailed description how to merge the
different approaches. 

\begin{algorithm}[Principal subfield] \label{principal_subfield} \ \\
{\bf Input:} {\rm $p$-adic factors $f_1,f_2$ of $f$. $p_s$-adic roots $r_1,\ldots,r_n$ of $f$.} \\
{\bf Output:} {\rm The principal subfield to any root of $f_2$.} 
\begin{enumerate}
\item
Choose an initial $p$-adic precision $\pr_L$.
\item 
Lift the factors $f_1,f_2$ to precision $\pr_L$.
\item
Build up the 
lattice as described in 
step 1 of  algorithm  Principal in~\cite{HKN}. 
\item
Apply LLL-with-removals~\cite[Algo. 2]{HN} 
to the lattice with the bound $n^2 \| f \|$.
This results in a $\bbQ$-vector subspace $U$ such that $K \subset U \subset L$.
\item
If $U$ is of dimension 1, return $\bbQ$ as principal subfield. 
\item
We denote by $h_1(\beta),\ldots,h_k(\beta)$ the basis 
of $U$ found.
\item
Set the block identify precision $\pr_B$ to $1$.
\item  \label{bl_id}
Compute the block values  
$V := \{ (h_1(r_i), \ldots, h_k(r_i)) \colon 1\leq i\leq n  \}$
with $p_s$-adic precision $\pr_B$.
\item \label{block_test_1}
If the set $V$ has less than $k$ elements, double the precision $\pr_B$ and redo 
the last step.
\item \label{block_test_2}
If $|V| > k$ or one tuple occurs less than
$n / k$ times go to step 2 with doubled precision $\pr_L$.
\item
Compute the potential block system 
$$
\calB := \left\{ \{ i \in \{1,\ldots,n\} \colon  (h_1(r_i), \ldots, h_k(r_i)) = v \} 
\colon v \in V \right\}\,.
$$
\item
Confirm the block system $\calB$ (see Remark \ref{block_direkt_beweisen}).
If this is successful, we get a defining polynomial of the subfield $g$ and 
a root $h(\beta) \in L$ of it.
If the confirmation fails, restart at step 2 with doubled precision $\pr_L$.
\item
Check that the roots of $f_1$ and $f_2$ are in the same block of the block system
corresponding to the subfield found.
If this fails, double precision $\pr_L$ and restart at step 2.
\item
Return the  primitive element $h(\beta)$ and its minimal polynomial $g$ as
principal subfield.
\end{enumerate}
\end{algorithm}

\begin{remark}
\begin{enumerate}
\item
Note that the $p$-adic arithmetic in the LLL-computation and the $p_s$-adic 
arithmetic for the block identification and the subfield confirmation are
independent. The interaction is only via $U$ and the subfield polynomial $g$ and its
root $h(\beta)$. None of these data is of $p$-adic nature. 
\item
When we get the subspace $U$ from the LLL-computation we can not assume that this 
is a subfield and we can not assume that it is the one we are targeting. 
We only know, that the subfield we are targeting is a vector subspace of $U$. 
Assuming $U$ to be a subfield, we can identify the corresponding block system. 
In the case that the outcome is coarser than the expected block system 
(step~\ref{block_test_1}) we increase the precision for the block identification.

If the outcome is finer than a block system (step~\ref{block_test_2}), we have 
proved that $U$ is not a subfield. 
\item
In the case that we successfully find a subfield, we are not yet done. If it does not 
relate to a block system which has the roots of $f_1$ and $f_2$ in the first block, 
the field found is larger than the one we target.

In the case that the last step is successful, we can conclude from the block analysis
only that the subfield found is a subfield of the targeted principal field. 
On the other hand $U$ is a vector subspace that contains the targeted principal 
field. As the dimension of $U$ and the subfield degree coincide we are done.
\end{enumerate}
\end{remark}

\begin{example} \label{A5_Beispiel}
Let $K / k$ be a degree 60 Galois extension with group $A_5$.
As the extension is Galois all  $f_i$ are linear. 
The first call to the principal subfield algorithm gives us a degree 12 subfield corresponding to
a block system of block size 5. The three remaining linear factors in the first block
of this block system are skipped.

The second call to the principal subfield algorithm gives us a second degree 12 subfield that is 
a conjugate of the first one. Now, we compute the intersection of the corresponding wreath products.
This is a group of order 240 isomorphic to an index 2 subgroup of $S_5 \times C_4$. 
It has 7 systems of blocks. 
As this group is not a subgroup of $A_{60}$, we can descent to the 
intersection $U$ with $A_{60}$ of order 120. 
$U$ is isomorphic to $A_{5} \times C_2$ and has 13 systems of blocks.

At this point, we determine all index two subgroups. As $A_{5}$ is simple, we find only one. 
It has 57 systems of blocks. 
After we confirmed one of the additional subfields by using Remark~\ref{block_direkt_beweisen} 
we have proved that the extension is regular with
Galois group $A_5$. Thus, we are done  with 2 LLL computations instead of 59.
\end{example}

\section{The lattice test}
Let us assume the following situation. Let $L,f$ be as above. 
Further, we have the factorization of $f$ modulo $p$ as $f_1 \cdots f_m$.
We assume that $f_1$ is linear. We denote the root of $f_1$ by $r_1$.
In addition to this, we know already the subfields $K_1,\ldots,K_\ell$ of $K$. 
We denote the block containing $r_1$ of the block systems corresponding to $K_1,\ldots,K_\ell$
by $\Delta_1, \ldots, \Delta_\ell$. As the $\Delta_i$ are Frobenius invariant, we can 
view each $\Delta_i$ as a subset of the above modulo $p$ factorization of $f$.
We want to check, if there is space for an additional principal subfield corresponding to the pair
$r_1, r_i$. Here, $r_i$ is a root of the local factor $f_j$.   
For this we use the algorithm of this section. We give the proof of correctness in 
Remark~\ref{remark_lat_test} and Lemma~\ref{pq_lemma}. 
Examples 
are given in Example~\ref{examples_lat_test}.

\begin{algorithm}[Lattice test] \label{lattice_test} \ \\
{\bf Input:} {\rm 
\begin{enumerate}
\item
The $p$-adic factorization of $f = f_1 \cdots f_m$.
\item
The index $j \in \{2,\ldots,m\}$ of the factor $f_j$ to be tested.
\item
A list of possible block sizes.
\item 
A list $\Delta_1,\ldots,\Delta_\ell$ of subsets of \{1,\ldots,m\} encoding the $p$-adic 
factors in the first blocks of known block systems.
\item
Labels indicating which $\Delta_i$ are proven to be principal.
\end{enumerate} 
{\bf Output:} {\rm `do factor' or `skip factor'} }
\begin{enumerate}
\item
Find the smallest first block $\Delta$ in $\Delta_1,\ldots, \Delta_\ell$ that contains $j$.
\item
In the case that no such first block is found set $\Delta := \{1,\ldots, m\}$.
(We are looking for a refinement of $\Delta$.)
\item 
Set $n_0 := \sum_{i \in \Delta} \deg f_i$. This is the block size of~$\Delta$. 
\item
Compute the list $N$ of all known refinements of $\Delta$. These are all the $\Delta_i$ with
$\Delta_i \subsetneq \Delta$.
\item
In the case that $n_0 = 4$ and $\Delta$ is known to be a principal block and we 
have one refinement in $N$ return `skip factor'.
\item
In the case that $n_0 = 8$ and $\Delta$ is known to be principal and $N$ contains 
a block of size 2 and a principal block of size 4, return `skip factor';
\item
From the list of potential block sizes extract the list $S$ of all entries that 
are proper divisors of $n_0$.
\item
Delete from $S$ all block sizes that can not be written as a sum of $1 + \deg f_j$ and 
the degrees of other factors in $\{ f_i \colon i \in \Delta \setminus \{1,j\} \}$.
\item
If a possible block size $d \in S$ satisfies 
$(d - \deg f_j) \cdot \frac{n_0}{k} < d$ 
with $k$ the block size of a known refinement in $N$, delete $d$ from $S$. 
\item
In the case that $n_0 = k \cdot q$ with a prime number $q > k$ and we know a 
refinement of $\Delta$ with block size $q$, delete $q$ from $S$.
\item
In the case that $n_0$ is the square of an odd prime number, do the following:
\begin{enumerate}
\item
In the case that there are two known refinements in $N$ and 
$\Delta$ is known to be principal return `skip factor'. 
\item
In the case that there are two known refinements in $N$ and
the factors $\{ f_i : i \in \Delta \setminus \{1\} \}$ 
are not all of the same degree return `skip factor'. 
\end{enumerate}
\item
In the case that $n_0 = p \cdot q$ with two 
prime numbers $q > p$ do the following:
\begin{enumerate}
\item
In the case that $q \not\equiv 1 \bmod p$ and $N$ contains a refinement 
with block size $p$, delete $p$ from $S$.
\item
In the case  $p \cdot q \not \in \{21, 55\}$ and a refinement with block size
$p$ is known do the following:
If $\Delta$ is known to be principal or $\Delta$ contains a factor of 
degree $> 1$, delete $p$ from $S$.
\item
In the case that $p \cdot q \in \{21, 55\}$ and we know two refinements in $N$ and
$\Delta$ is known to be principal, return `skip factor'.
\end{enumerate}
\item
In the case that $S$ is empty return `skip factor', otherwise return `do factor'.
\end{enumerate}
\end{algorithm}

\begin{remark} \label{remark_lat_test}
The above test uses the following facts about block systems.
\begin{enumerate}
\item
If a block system refines another block system, the block size 
of the first one divides the block size of the second one.
\item
Given two block systems, the non-empty intersections 
$B_i \cap B'_j$ are all of the same size.
\item
In degree $4$ the only permutation group with more than 
one (non-trivial) system of blocks is the Kleinian four group.
\item
A degree $8$ permutation group with principal blocks of size 2, 4, 8 has
no other blocks.
\item
In degree $p^2$ ($p$ is an odd prime) a transitive group has either at most 
two systems of blocks or it is contained in $(C_p \times C_p) \rtimes C_{p-1}$. 
In the latter case it has $p+1$ systems of blocks and $\{1,\ldots,p^2\}$
is not principal. Further, all cycles with a fixed point are of the type 
$(1,d,\ldots,d)$ with a divisor $d$ of $p-1$.
\item
In degree $k \cdot q$ with $q > k$ ($q$ a prime number) 
there is at most one system of blocks with block size~$q$.
\item
In degree $p \cdot q$ with $q > p$ ($p$,$q$ both prime numbers) and $q \not\equiv 1 \bmod~ p$ 
there is at most one system of blocks with block size~$p$.
\end{enumerate}
Some of the above statements are well known. The $p^2$-case is worked out in~\cite{DW}. As
we do not know a good reference for the $(p \cdot q)$-case we give the proof below.
\end{remark}

\begin{lemma}  \label{pq_lemma}
Let $p,q$ with $p < q$ be prime numbers and $G \subset \Sym_{pq}$ be transitive
with more than one block system of block size~$p$.
Then $q \equiv 1\, (\bmod\, p)$. 

If $G$ is solvable then $G$ is the regular representation of 
$\bbZ / q \bbZ \rtimes\bbZ / p \bbZ$ with exactly $q$ block systems of 
block size $p$. The trivial block system with only one block is not principal.
 
If $G$ is not solvable we have $p \mid q - 1 \mid p (p-1)$. 
Assuming the classification of finite simple groups
we have one of the following two cases:
\begin{enumerate}
\item
$p =  3$, $q = 7$, $\#G = 168$ and $G \cong {\rm GL}_3(\bbF_2)\cong {\rm PSL}_2(\bbF_7)$.
\item
$p = 5$, $q = 11$, $\#G = 660$ and $G \cong {\rm PSL}_2(\bbF_{11})$. 
\end{enumerate}
Both cases result in exactly two block systems of block size~$p$ 
and none of block size~$q$.
\end{lemma}

\begin{proof}
Let $\calB_1, \calB_2$ be two block systems with block size $p$. 
Thus, we have $G \leq S_p \wr S_q$ and $q \| \# G$.
We denote the corresponding
$G$-action on the blocks by $\varphi_1$ and $\varphi_2$. The intransitive representation of $G$ given
by $\varphi_1 \oplus \varphi_2$ is faithful.  Thus, $G$ is a sub-direct product in 
$\varphi_1(G) \times \varphi_2(G)$.  

As $\#G$, $\#\varphi_1(G)$ and $\#\varphi_2(G)$ are exactly once divisible by $q$ and each
nontrivial normal subgroup of a transitive permutation group of prime degree $q$ contains 
all the $q$-Sylow-subgroups \cite[Kap.\ II Satz 1.5]{Hup} we can conclude that 
$\varphi_1$ and $\varphi_2$ are faithful permutation representations of $G$.

In the case that $G \cong \varphi_1(G)$ is solvable we can use the result of Galois 
$\varphi_1(G) \cong (\bbZ / q \bbZ) \rtimes (\bbZ / r \bbZ)$~\cite[Kap. 2, Satz 3.6]{Hup}. 
Here,  $r$ is a divisor of $q - 1$. As $G$ acts transitive on $p \cdot q$ 
points we conclude $p \mid r \mid q-1$. In the case that $r = p$ the group $G$ 
has exactly $q$ Sylow $p$-subgroups and one Sylow $q$-subgroup.
Thus, we have $q$ block systems with block size $p$ and one with block size $q$. 
If $p < r$ the point stabilizer in $G$ is a cyclic group of order $\frac{r}{p}$. 
The unique index $q$ subgroup that contains it is its normalizer. 
Thus, $G$ has only one block system of block size $p$.

Now we assume that $G \cong \varphi_1(G)$ is not solvable. 
By a theorem of Burnside~\cite[Kap. 5, Satz 21.3]{Hup} $\varphi_1(G)$   
is 2-transitive. Let $U \subset G$ be the stabilizer of one block of $\calB_1$. 
The action of $U$ on $\calB_1$ has one orbit of size $1$ and one of size $q-1$.
Thus, the action of $U$ on the $pq$ points has one orbit of length $p$
and the lengths of the other orbits are multiples of $q-1$. 

The orbit of length $p$ consists of the points in the block stabilized. 
It hits $p$ blocks of $\calB_2$. 
These blocks form a $U$-invariant set of order $p^2$. Thus, $p^2$ is the sum of $p$ and
some multiples of $q-1$. This shows $(q-1) \mid (p^2 - p) = p (p-1)$. As $q$ is bigger than $p$
we can conclude $p \mid q-1 \mid p (p-1)$. 

As $G \cong \varphi_1(G)$ is of prime degree $q$ and not solvable the classification 
of finite simple groups implies that $\varphi_1(G)$ is one of:~\cite[Sec. 3.5]{DM}
\begin{itemize}
\item[1)]
The symmetric or the alternating group of degree $q$.
\item[2)]
${\rm PSL}_2(\bbF_{11})$ acting on $q = 11$ points.
\item[3)]
The Mathieu group $M_{q}$ with $q = 11$ or $q = 23$.
\item[4)]
A projective linear group $G$ with ${\rm PSL}_d(\bbF_\ell) \subset G \subset {\rm P \Gamma L}_d(\bbF_\ell)$
of degree $q = \frac{\ell^d-1}{\ell-1}$. Here, $d$ is a prime.
\end{itemize}
Further, we have proved that the one-point stabilizer of $\varphi_1(G)$ has an index $p$ subgroup with
$p \mid q-1 \mid p (p-1)$. 
We get
$$
q \in \{q \in \bbP \mid \exists p \in \bbP \colon p \mid q-1 \mid p (p-1)\}
= \{ 3, 7, 11, 23, 43, 47, 53, \ldots \}\, .
$$
As $S_3$ is solvable, $q$ is at least 7.
Now we check, which of the above options is compatible with this.
\begin{itemize}
\item[1)]
  The one-point stabilizer of the symmetric and the alternating groups of degree $q \geq 7$ is
  the symmetric or the alternating group of degree $q-1$. The subgroups of smallest index 
have index 2 or $q-1$. They are given by the intersection with the alternating group or 
as the stabilizer of a second point. Further, if $q-1 = 6$ we get another index 6 subgroup. 
This excludes the option of a suitable index $p$ subgroup.
\item[2)]
The one-point stabilizer of $\varphi_1(G) = {\rm PSL}_2(\bbF_{11})$ has exactly one maximal subgroup of 
prime degree. It results in one of the exceptions listed in the claim.
\item[3)]
The one-point stabilizers of the Mathieu groups have only one subgroup of prime index.
This is an index 2 subgroup in the one-point stabilizer of $M_{11}$.
This excludes the option.
\item[4)]
Here we have $p \mid \frac{\ell^d-1}{\ell-1} - 1 \mid p (p-1)$. In the special case 
of $d = 2$ we get $p \mid \ell \mid p (p-1)$. As $p$ is a prime number and $\ell$ a 
prime power we conclude $p = \ell$ and $q = p + 1$.
This forces $p = 2$, $q = 3$ in contradiction to $q \geq 7$.

Otherwise $d = 2 k + 1$ is an odd prime number. This results in 
\begin{eqnarray*}
p \, \bigg| \, q-1 &=& \ell (\ell^{2k-1} + \ell^{2k-2} + \cdots + \ell + 1)  \\
  &=& \ell (\ell^k + 1) (\ell^{k-1} + \ell^{k-2} + \cdots + \ell + 1)\, \bigg| \, (p-1) p.
\end{eqnarray*}
As $q-1$ is larger than $\ell^{2k} + \ell^{2k-1}$ we get the estimate 
$$
\ell^{2k} + \ell^{2k-1} < q-1 \leq p(p-1) < p^2.
$$
This implies $p \geq \ell^k + 1$. Here, equality is only possible in the case of $k = 1$. 
As the largest factor of $q-1$ is $\ell^k + 1$, we can conclude $p = \ell^k + 1$, $k = 1$, and $d = 3$. 

Using once more that $\ell$ is a prime power we get that $p$ is a 
Fermat prime and $\ell$ has the shape $2^{2^e}$. 
As $q = \ell^2 + \ell + 1$ we can conclude $3 | q$ as long as $2^e$ is even. 
Thus, the last remaining possibility is $e = 0,$ $\ell = 2,$ $p = 3,$ $q = 7$, and $d = 3$. 
Is results in $\varphi_1(G) = {\rm GL}_3(\bbF_2)$ as a potential group. This is the other exception 
in the claim. \qedhere
\end{itemize}
\end{proof}

\begin{examples} \label{examples_lat_test}
When applying the lattice test to a root $r_i$ of a local factor $f_j$ of the field extension $K / \bbQ$
the algorithm will first determine the largest known subfield $k$ such that $r_1$ and $r_i$ 
are in the same block. Then our criteria will be applied to the relative extension $K / k$. 
To keep the examples simple, we will describe only the relative situation:
\begin{itemize}
\item[1)] Assume that $K/ k$ is a cyclic extension of degree 15. 
As the Galois group is a regular permutation group, we get 15 linear factors $f_j$. We need two successful
calls to the principal subfield algorithm to get the subfields of degree 3 and 5. After this, the 
$p \cdot q$-case of the lattice test applies and stops the search for further principal subfields. 
Note, that the intersection of the wreath products will result in the group $S_3 \times S_5$. Using 
the discriminant we can descent to an index 2 subgroup. 
\item[2)] Let $K/ k$ be a cyclic extension of degree 10. 
As above, all $f_j$ are linear. Here we have to find the degree 2 and the degree 5 subfield with one 
successful call to the principal subfield algorithm for each. Further, we have to confirm that $k$ is principal. 
Otherwise the Galois group could be the regular representation of a dihedral group with 4 other subfields. 
Thus, in the best case 3 LLL-calls suffice.
\item[3)] Assume that $K / k$ is a Galois extension of degree 10 with a dihedral Galois group. Here 
we have to determine two block systems with block size 2. Then the intersection of the wreath products 
descents to a group of order 10 and gives us 4 more subfields. Knowing all these subfields the lattice
test will stop the search. Aside from this, the lower bound of the group order could be used. 
\item[4)] Let $K / k$ be a degree 21 extension with Galois group $21T4 \cong C_7 \rtimes C_6$. 
Here, we can find a prime with three linear and nine quadratic local factors. 
A first call to the principal subfield algorithm with a linear local factor will result in a
degree 7 subfield. Now, the third linear factor can be skipped as it is part of a known block of size 3.

The principal subfield algorithm applied to 3 of the 9 quadratic factors will 
give us the last missing block system of block size 7 corresponding to a cubic subfield.   
Thus, one successful call to the principal subfield algorithm suffices. 
At this point the $p \cdot q$ test applies. It concludes that the only 
possibility of even more subfields corresponds to the regular Galois 
group $C_7 \rtimes C_3$. As we started with an irregular cycle, this
option is excluded. Thus, we have all subfields.
\item[5)] 
Assume $K / k$ to be a degree 8 extension with Galois group $8T17$ of order 32. 
We can choose a prime such that four of the $f_i$ are linear and the fifth is quartic. 

As no block system with block size bigger than 4 is possible, we can skip the quartic factor
and label $k$ as principal. Applying the principal subfield algorithm twice to linear factors 
gives us a quadratic and a quartic principal subfield. Now, we can skip the last linear factor. 
This follows from the statement on degree 8 permutation groups as we have principal subfields of
degree 1,2,4. 

\end{itemize}
\end{examples}

\section{Applications to Galois group computation}\label{sec_applications}

The computation of the Galois group of the splitting field of a polynomial
$f$ of degree $n$ with rational coefficient is usually done by a method 
introduced by Stauduhar~\cite{St}. See~\cite{Ge,FK} for further  
developments. It constructs a descending chain of 
subgroups starting at a sufficiently large group (e.g. $\Sym_n$) down to 
the Galois group.

To perform well in practice, we have to start with a group as small as 
possible. Subfields are used at this point, as each of them relates
to a wreath product that contains the Galois group. 
Initially, the Galois group of each subfield was determined, to get an
even smaller wreath product. Finally the intersection of all
 wreath products was determined~\cite[Algorithmus 5.3]{Ge}.

Doing it this way results in two disadvantages. 
Subfields that are contained in several maximal subfields were 
treated multiple times in the recursive approach.   
Second, the intersection of several wreath products can be a very small group. 
Thus, the strategy should be rearranged. We present it as an algorithm. We remark that
we perform step~2 using the graph theoretic methods from Section \ref{sec:int}.
We have the hope that this step is in polynomial time, but we can not prove this.

\begin{algorithm}[Galois starting group]\label{start} \ \\
{\bf Input: }{\rm A field $L$.} \\
{\bf Output: }{\rm An overgroup of the Galois group of $L$.}
\begin{enumerate}
\item
Determine Galois-generating subfields of $L$. 
\item
Set $G_0$ to the 
intersection of the wreath products corresponding to these fields.
\item
Determine the  projection of $G_0$ to a starting group of the Galois group of
$\bbQ(\{ \sqrt{\disc(K_i)} : i \in I \}) / \bbQ$.
Here, $\{ K_i : i \in I \}$ is the set of all subfields of $L$.
\item 
Determine the Galois group of the multi-quadratic extension above and 
replace $G_0$ by the pre-image of this group.
\item
Find the subfield $K$ of smallest degree such that the projection of $G_0$ 
to a Galois starting group of $K$ does not result in a group of order equal
to the lower bound of the Galois group of $K$ computed by cycle type inspection.
\item 
Determine the projection $\pi_K$ that maps $G_0$ to its action on the block system
corresponding to $K$.
\item
Use $\pi_K(G_0)$ as a starting group to 
determine the group $G_K$ of $K$ by the Stauduhar method.
\item
Replace $G_0$ by $\pi_K^{-1}(G_K)$.
\item
Redo the last three steps for all other subfields $K$ with degrees in 
ascending order. 
\item \label{rel_disc}
If $L / \bbQ$ has a unique maximal subfield $K$ and $[L:K] > 2$ 
compute the field $K_\Delta := K[\disc(L / K)]$ and the homomorphism $\pi$
that maps $G_0$ to a starting group for $K_\Delta$.
\item \label{rel_des}
In the case that $K_\Delta$ was constructed, determine the Galois group $G_\Delta$ 
by the Stauduhar method and replace $G_0$ by $\pi^{-1}(G_\Delta)$.
\item
Return $G_0$ as a starting group for the Galois group computation of $L$.
\end{enumerate}
\end{algorithm}

\begin{remark}
In step \ref{rel_disc} we use an auxiliary field that is not part of
the subfield lattice. This is only done in the case that the relative 
degree is bigger than 2 and the maximal subfield is unique. The reason for 
these restrictions is that otherwise this subfield would coincide with $L$ 
or the intersection of the wreath products together with the subfield 
inspection results already in a small starting group. 
Thus, in the latter case there is no reason to
expect a large index descent in step~\ref{rel_des}.

A systematic test with the database of transitive groups brings the
degree 42 group with number 5798 to light. It relates to a degree 42
field with maximal subfields of degree 14 and 21. The intersection of the 
two maximal subfields is of degree 7. The subfield groups are $\Sym_2 \wr \Sym_7$
and $\Sym_3 \wr \Sym_7$. 
Thus, in this case the computed starting group is 
$D_{2 \cdot 6} \wr \Sym_7$.
Here, step~\ref{rel_disc} would construct
a degree 28 field with starting group $(\Sym_2 \times \Sym_2) \wr \Sym_7$. 
Applying step~\ref{rel_des} would result in an descent to a maximal 
subgroup of index 64.

\end{remark}

\begin{example}
Lets illustrate this strategy by applying it to $f_{18} = x^{18} + 9x^9 + 27$.
The subfield algorithm determines the subfields given by 
$f_2 := x^2 + 9 x + 27, f_3 := x^3 - 12 x^2 + 39 x - 37$, and $f_6 := x^6 + 9 x^3 + 27$. 
The sextic field is the composition of the smaller ones. 
The subfields correspond to the systems of blocks 
\begin{eqnarray*}
&\left\{ \{ 1, 2, 3, 4, 5, 6, 7, 8, 9 \}, 
         \{ 10, 11, 12, 13, 14, 15, 16, 17, 18 \} \right\}, \\
&\left\{ \{ 1, 2, 3, 13, 14, 15 \}, \{ 4, 5, 6, 16, 17, 18 \}, 
         \{ 7, 8, 9, 10, 11, 12 \} \right\}, \\
&\left\{ \{ 1, 2, 3 \}, \{ 4, 5, 6 \}, \{ 7, 8, 9 \}, 
         \{ 10, 11, 12 \}, \{ 13, 14, 15 \}, \{ 16, 17, 18 \} \right\}  .
\end{eqnarray*}
The intersection of the corresponding wreath products results in the group 
18T903 of order 559872. This projects to a starting group for the sextic 
subfield of order $12$. As the discriminant of $f_3$ is a square 
we can pass to an index 2 subgroup. This determines the groups of all 
subfields.

As $f_2$ and $f_{18}$ have the same discriminant up to squares, we can do another
descent to an index 2 subgroup. This gives us the next smaller
starting group 18T806.

At this point the algorithm determines that the Galois group  of $\bbQ[X] / f_{18}$ viewed as a degree
3 extension of $\bbQ[X] / f_{6}$ is even 
as the resolvent $x^{12} + 531441 x^6 + 282429536481$
is reducible. Thus, we return the group
18T453 of order 4374 as starting group for the Galois group computation.

In the case that we want to finish the computation of the Galois group of $f_{18}$,
we descent via 18T323, 18T160 and 18T82 to 18T16. All descents are of index 3.
\end{example}

\section{Algorithm selection}

At a first glance it seems that the non-polynomial time algorithm~\cite{Kl} 
is now outdated. However, this is not the case. The point is, that there are 
many cases, where the old algorithm had to do only very few steps, each of 
them being a lot simpler than an LLL-reduction.

This can be explained by interpreting the old algorithm 
as part of a Galois group computation with the Stauduhar method.
This method determines the Galois group by constructing a descending chain
of overgroups of the Galois group till the Galois group is reached.

More precisely, determining a subfield of degree $\ell$ is equivalent to 
determining a wreath product $\Sym_k \wr \Sym_\ell \subset \Sym_n$
that contains the Galois group. 
The number of subgroups conjugate to $\Sym_k \wr \Sym_\ell$ is
$\binom{n-1}{k-1} \binom{n-k-1}{k-1}  \cdots \binom{2k -1}{k-1} \binom{k -1}{k-1}$.

Testing all these conjugates is too costly. 
The starting point of~\cite{Kl} is to work with $p$-adic root approximations in an
unramified $p$-adic extension. Then, the local Galois group is generated by the 
Frobenius element, which is known to be an element of the Galois group. 
In many cases only very few conjugates of $\Sym_k \wr \Sym_\ell \subset \Sym_n$
contain a given element $\Frob \in \Sym_n$.
In other words, there are only very few potential block systems that are compatible 
with the Frobenius action on the roots.

In the more general setting of Galois group computations 
the conjugate subgroups of 
a subgroup that contain a prescribed Frobenius element are 
determined by using {\em short cosets} which are formed by a the 
coset representatives $\sigma$ of $G / U$ such that $\Frob \in U^\sigma$~\cite{Ge}.  
There are several ways to determine short cosets efficiently~\cite{E1}.

The number of short cosets can be derived from the cycle type of the
Frobenius permutation. In Table~\ref{sc_tab} we give a few examples for a regular cycle.

\begin{table}[h]
\begin{center}
\begin{tabular}{c|c|c}

Cycle type     & Possible cycle types  & Number of        \\
of Frobenius   & on blocks             & short cosets     \\
\hline  
$(n)$                                    &          $(\ell)$ &   1 \\
$\left(\frac{n}{2}, \frac{n}{2}\right)$             & 
$(\ell), \left(\frac{\ell}{2}, \frac{\ell}{2}\right)$ &                   $\ell, 1$ \\[0.5mm]    
$\left(\frac{n}{3}, \frac{n}{3}, \frac{n}{3}\right)$&
$(\ell), \left(\frac{\ell}{3}, \frac{2 \ell}{3}\right),
\left(\frac{\ell}{3}, \frac{\ell}{3}, \frac{\ell}{3}\right)$&    $\ell^2, 2\ell, 1$ \\[0.5mm]
$\left(\frac{n}{4}, \frac{n}{4}, \frac{n}{4}, \frac{n}{4}\right)$ &
$(\ell), 
\left(\frac{\ell}{2}, \frac{\ell}{2} \right), 
\left(\frac{\ell}{4}, \frac{3 \ell}{4} \right),$ 
& 
$\ell^3, \frac{3}{4} \ell^2, \frac{9}{4} \ell^2, $ \\[0.5MM]
&
$
\left(\frac{\ell}{4},\frac{\ell}{4},\frac{\ell}{2}\right),
\left(\frac{\ell}{4},\frac{\ell}{4},\frac{\ell}{4},\frac{\ell}{4}\right)$ 
& $3 \ell, 1$ \\[0.5mm]
$(1,\ldots,1)$ & $(1,\ldots,1)$ & 
$
\frac{n!}{n (n-k)(n-2k) \cdots (2k) k}
$
\end{tabular}
\smallskip

\caption{Number of short cosets for $\Sym_k \wr \Sym_\ell  \subset \Sym_{k \ell}$ \label{sc_tab}}
\end{center}
\end{table}

The table shows that the algorithm described in~\cite{Kl} should be used in 
the case that a Frobenius element with a small number of orbits is found.

\section{Practical performance test}
All test are carried out by using \magma~2.23 on one core of an 
Intel i5 
running at 3.5 GHz.
For testing the method we use the polynomials listed in: 

\centerline{\tt http://www.math.fsu.edu/$\sim$hoeij/subfields/} 

It includes the runtime for an Intel core 2 running at 2.33 GHz.
The  test results are listed in Table~\ref{test_tab}. The column total time gives 
the time to construct the entire subfield lattice. In the examples 
12, 14, 15, and 20 a significant amount  of time (5.7/17.6/29.9/318.5 sec.) 
is used to construct the subfield lattice out of the generators found. 

\begin{table}[h]

\begin{center}
{\scriptsize
\begin{tabular}{r|c|c|c|c|c|c|c|c|c}
 Nr & Degree & Divisor     &  Galois group   &  $\#G_0$ & LLL          & LLL       & total   & time for   & time given in\\ 
    &        & of $\#\Gal$ &                 &          &  calls       & time      & time    & \cite{Kl} & \cite{HKN}\\
\hline
  1 &    36  &    108      & ${\rm C}_3 \times \Sym_3 \times \Sym_3$
                                             &   108    &       10      &   1.63   &   2.19  &   0.43 & 30.69  \\
  2 &    75  &    300      & ${\rm C}_5^2 \rtimes {\rm C}_{12} $
                                             &   300    &        8      &  84.66   &  87.03  &  26.39 & 453.24 \\
  3 &    48  &    192      & ${\rm C}_4^2 \rtimes {\rm D}_{2 \cdot 6}$  
                                             &   384    &        7      &   4.72   &   5.98  &   6.16 & 101.24 \\
  4 &    56  &    336      &${\rm Aut}({\rm GL}_3(\bbF_2))$
                                             &   336    &        3      &   5.30   &   5.98  &  1415.99 & 192.17 \\
  5 &    50  &    100      &${\rm C}_5^2 \rtimes {\rm C}_4 $ 
                                             &   200    &        3      &   3.15   &   3.78  &  99.19 & 95.28 \\
  6 &    60  &    120      &    $\Sym_5$     &   120    &        5      &   14.04  &  16.77  &  98.07 & 457.97\\
  7 &    64  &    576      &${\rm C}_2^4 \rtimes (\Sym_3 \times \Sym_3)$ 
                                             &  2304    &       19      &   84.09  &  86.40  &  105.47 & 417.05 \\
  8 &    72  &    288      & ${\rm C}_6^2 \rtimes {\rm D}_{2 \cdot 4}$
                                             &   288    &        6      &   16.51  &  21.03  &  3934.64 & 785.93\\
  9 &    60  &     60      &    ${\rm A}_5$  &   120    &        2      &    9.55  &  16.54  &  2442.16 & 746.17\\
 10 &    81  &    162      & ${\rm C}_3 \times ({\rm C}_3^3 \rtimes {\rm C}_2) $    
                                             &   324    &        5      &   68.58  &  79.93  & 7405.01 & 1849.31 \\
 11 &    81  &    162      &${\rm C}_3^2 \wr \Sym_2$
                                             &   324    &        8      &  126.73  & 135.93  &  13945.61 & 1644.12 \\
 12 &    32  &     32      &${\rm C}_2^5$    &    32    &        5      &    0.39  &   6.64  &  $> 50$h & 170.80 \\
 13 &    64  &    256      & ${\rm D}_{2 \cdot 4} \times ({\rm C}_{4}  \rtimes {\rm D}_{2 \cdot 4}) $
                                             &   512    &        9      &   26.41  &  34.13  & $> 50$h & 609.11  \\
 14 &    96  &  11520      &$\Sym_6 \times {\rm C}_2^4$          
                                             & 11520    &        6      &  110.48  & 142.30  & $> 50$h & 5494.3 \\
 15 &    96  &     96      &${\rm C}_2^4 \rtimes {\rm C}_6$
                                             &    96    &        3      &   62.05  & 110.04  & 42395.95 & 18440.68 \\
 16 &    75  &    300      & ${\rm C}_5^2 \rtimes ({\rm C}_3 \rtimes {\rm C}_4)$          
                                             &   600    &        3      &   72.32  &  77.69  & 93621.32 & 2911.19 \\
 17 &    80  &    160      & ${\rm C}_2^4 \times {\rm D}_{2\cdot 5}$           
                                             &   320    &        6      &   92.47  & 115.56  & $> 50$h & 4195.90\\
 18 &    90  &    360      &    ${\rm A}_6$  &   360    &        8      &  292.22  & 298.09  & $> 50$h & 4053.24\\  
 19 &   100  &    100      & Equal to Nr.~5  &   100    &        3      &  189.00  & 224.83  & $> 50$h & 43376.56\\
 20 &    64  &     64      &${\rm C}_2^6$    &    64    &        6      &   26.87  & 354.06  & $> 50$h & ---
\end{tabular} }
\smallskip

\caption{Performance of the subfield algorithms (all times in seconds) \label{test_tab}}

\end{center}
\end{table}

We can read from the table that except for the first two examples, the new 
method is faster than~\cite{Kl}. Further, the lower bound of the order of the Galois
group is sharp in all examples. The row with $\#G_0$ lists the order of
the Galois starting group in the moment we stop calling the 
principal subfield algorithm. 

For the LLL-based approach, our computation is a factor of \mbox{$4.8$ -- 190} faster than~\cite{HKN}.
Following {\tt http://cpu.userbenchmark.com}, 
our hardware is about a factor of 2.3 faster (single-core int speed).
 
Taking the newer hardware into account, the implementation is a factor 
\mbox{$2$ -- 80}
better. The improvement can partially be explained by the
faster subfield proof and a potentially better LLL implementation. 
The main reason for the speed up is the reduced number of LLL-calls.
In the degree 64 example Nr.~7 we have the smallest improvement. The number of 
$p$-adic factors is 24. We have to call the LLL for 19 of them.
In the degree 100 example, we have 100 $p$-adic factors but the LLL is called only 3 times.  
This explains why we have the biggest improvement here.

To go even further,
we would need a better strategy to select the $p$-adic factor treated 
next and a good heuristic to stop the principal subfield search
even earlier and finish with another (e.g.\ the Stauduhar) method.

\end{document}